\documentclass[reqno,dvips]{amsart}

\usepackage{mathrsfs}
\usepackage{amsmath}
\usepackage{amssymb}
\usepackage{cite}
\usepackage{latexsym}
\usepackage{graphicx}

\usepackage{color,comment}

\theoremstyle{plain}
\newtheorem{thm}{Theorem}[section]
\newtheorem{lem}[thm]{Lemma}

\newtheorem{dfn}[thm]{Definition}
\newtheorem{prop}[thm]{Proposition}
\newtheorem{rmk}[thm]{Remark}

\def\D{\mathrm{D}}

\def\R{\mathscr{R}}
\def\T{\mathrm{T}}

\def\d{\mathrm{d}}

\def\r{\mathrm{R}}

\def\Cset{\mathbb{C}}
\def\Nset{\mathbb{N}}
\def\Qset{\mathbb{Q}}
\def\Rset{\mathbb{R}}

\def\Zset{\mathbb{Z}}

\def\Im{\mathrm{Im}}
\def\Re{\mathrm{Re}}
\def\Span{\mathrm{span}}

\def\epsilon{\varepsilon}

\makeatletter
\@addtoreset{equation}{section}
\makeatother
\def\theequation{\arabic{section}.\arabic{equation}}

\begin{document}


\title[Nonintegrability near degenerate equilibria]%
{Nonintegrability of dynamical systems near degenerate equilibria}

\author{Kazuyuki Yagasaki}

\address{Department of Applied Mathematics and Physics, Graduate School of Informatics,
Kyoto University, Yoshida-Honmachi, Sakyo-ku, Kyoto 606-8501, JAPAN}
\email{yagasaki@amp.i.kyoto-u.ac.jp}

\date{\today}

\begin{abstract}
We prove that general three- or four-dimensional systems 
 are real-analytically nonintegrable near degenerate equilibria in the Bogoyavlenskij sense
 under additional weak conditions
 when the Jacobian matrices have a zero and pair of purely imaginary eigenvalues
 or two incommensurate pairs of purely imaginary eigenvalues at the equilibria.
For this purpose, we reduce their integrability
 to that of the corresponding Poincar\'e-Dulac normal forms
 and further to that of simple planar systems,
 and use a novel approach for proving the analytic nonintegrability of planar systems. 
Our result also implies that general three- and four-dimensional systems
 exhibiting fold-Hopf and double-Hopf codimension-two bifurcations, respectively,
 are real-analytically nonintegrable under the weak conditions.
To demonstrate  these results,
 we give two examples for the R\"ossler system and coupled van der Pol oscillators.
\end{abstract}
\maketitle


\section{Introduction}

In this paper we study the nonintegrability of systems of the form
\begin{equation}
\dot{x}=f(x),\quad
x\in\Rset^n,
\label{eqn:sys}
\end{equation}
where $n=3$ or $4$ and $f(x;\mu)$ is analytic.
We assume that $x=0$ is an equilibrium, i.e., $f(0)=0$,
 and the Jacobian matrix $\D f(0)$ of $f(x)$ at $x=0$ has
 (I) a zero and pair of purely imaginary eigenvalues, $\lambda=0,\pm i\omega$ ($\omega>0$), for $n=3$
 or (II) two pairs of  purely imaginary eigenvalues, $\pm i\omega_j$ ($\omega_j>0$), $j=1,2$,
 with $\omega_1/\omega_2\not\in\Qset$ for $n=4$.
Here we adopt the following concept of integrability in the Bogoyavlenskij sense \cite{B98}.

\begin{dfn}[Bogoyavlenskij]
\label{dfn:1a}
For any integer $n\ge 1$,
 the $n$-dimensional system \eqref{eqn:sys}
 is called \emph{$(m,n-m)$-integrable} or simply \emph{integrable}
 for some integer $m\in[1,n]$ 
 if there exist $m$ vector fields $f_1(x)(:= f(x)),f_2(x),\dots,f_m(x)$
 and $n-m$ scalar-valued functions $F_1(x),\dots,F_{n-m}(x)$ such that
 the following two conditions hold$:$
\begin{enumerate}
\setlength{\leftskip}{-1.8em}
\item[\rm(i)]
$f_1(x),\dots,f_m(x)$ are linearly independent almost everywhere
 and commute with each other,
 i.e., $[f_j,f_k](x):=\D f_k(x)f_j(x)-\D f_j(x)f_k(x)\equiv 0$ for $j,k=1,\ldots,m$,
 where $[\cdot,\cdot]$ denotes the Lie bracket$\,;$
\item[\rm(ii)]
The derivatives $\D F_1(x),\dots, \D F_{n-m}(x)$ are linearly independent almost everywhere
 and $F_1(x),\dots,F_{n-m}(x)$ are first integrals of $f_1, \dots,f_m$,
 i.e., $\D F_k(x)^\T f_j(x)\linebreak\equiv 0$ for $j=1,\ldots,m$ and $k=1,\ldots,n-m$,
 where the superscript `{\small\,$\T$}' represents the transpose operator.
\end{enumerate}
We say that the system is \emph{analytically} $($resp. \emph{meromorphically}$)$
 \emph{integrable}
 if the first integrals and commutative vector fields are analytic $($resp. meromorphic$)$. 
\end{dfn}
 
If an $\ell$-degree-of-freedom Hamiltonian system with $\ell\ge 1$
 is  integrable in the Liouville sense \cite{A89,M99},
 then so is it in the Bogoyavlenskij sense,
 since it has not only $\ell$ functionally independent first integrals
 but also $\ell$ linearly independent commutative (Hamiltonian) vector fields
 generated by the first integrals.
Thus, the Bogoyavlenskij-integrability in Definition~\ref{dfn:1a}
 is considered as a generalization of Liouville-integrability for Hamiltonian systems.

Under our assumptions, by power series changes of coordinates,
 the system \eqref{eqn:sys} is formally transformed to
\begin{equation}
\begin{split}
\dot{x}_1=&-\omega x_2+g_1(x_1^2+x_2^2,x_3)x_1-g_2(x_1^2+x_2^2,x_3)x_2,\\
\dot{x}_2=&\omega x_1+g_2(x_1^2+x_2^2,x_3)x_1+g_1(x_1^2+x_2^2,x_3)x_2,\\
\dot{x}_3=&g_3(x_1^2+x_2^2,x_3)
\end{split}
\label{eqn:pd1}
\end{equation}
with $x=(x_1,x_2,x_3)$ for case~(I), and to 
\begin{equation}
\begin{split}
\dot{x}_1=&-\omega_1 x_2+h_1(x_1^2+x_2^2,x_3^2+x_4^2)x_1-h_2(x_1^2+x_2^2,x_3^2+x_4^2)x_2,\\
\dot{x}_2=&\omega_1 x_1+h_2(x_1^2+x_2^2,x_3^2+x_4^2)x_1+h_1(x_1^2+x_2^2,x_3^2+x_4^2)x_2,\\
\dot{x}_3=&-\omega_2 x_4+h_3(x_1^2+x_2^2,x_3^2+x_4^2)x_3-h_4(x_1^2+x_2^2,x_3^2+x_4^2)x_4,\\
\dot{x}_4=&\omega_2 x_3+h_4(x_1^2+x_2^2,x_3^2+x_4^2)x_3+h_3(x_1^2+x_2^2,x_3^2+x_4^2)x_4
\end{split}
\label{eqn:pd2}
\end{equation}
with $x=(x_1,x_2,x_3,x_4)$ for case~(II),
 where $g_j(y_1,y_2)$, $j=1,2,3$, and $h_j(y_1,y_2)$, $j=1,2,3,4$,
 are formal power series of $y_1$ and $y_2$, which may not be convergent,
 such that $g_j(0,0),\D_{y_2}g_3(0,0)=0$, $j=1,2,3$, and $h_j(0,0)=0$, $j=1,2,3,4$.
See, e.g., Section~3.1 of \cite{HI11} for the derivation of \eqref{eqn:pd1} and \eqref{eqn:pd2}.
Equations \eqref{eqn:pd1} and \eqref{eqn:pd2} are, respectively, represented as
\begin{equation}
\begin{split}
\dot{x}_1=&-\omega x_2+\alpha_1x_1x_3-\alpha_2 x_2x_3,\\
\dot{x}_2=&\omega x_1+\alpha_2 x_1x_3+\alpha_1 x_2x_3,\\
\dot{x}_3=& \alpha_3(x_1^2+x_2^2)+\alpha_4x_3^2
\end{split}
\label{eqn:pd1t}
\end{equation}
up to $O(|x|^2)$ for case~(I), and as
\begin{equation}
\begin{split}
\dot{x}_1
=& -\omega_1 x_2+(\alpha_1(x_1^2+x_2^2)+\alpha_2(x_3^2+x_4^2))x_1\\
& -(\beta_1(x_1^2+x_2^2)+\beta_2(x_3^2+x_4^2))x_2,\\
\dot{x}_2
=& \omega_1 x_1+(\beta_1(x_1^2+x_2^2)+\beta_2(x_3^2+x_4^2))x_1\\
& +(\alpha_1(x_1^2+x_2^2)+\alpha_2(x_3^2+x_4^2))x_2,\\
\dot{x}_3
=& -\omega_2 x_4+(\alpha_3(x_1^2+x_2^2)+\alpha_4(x_3^2+x_4^2))x_3,\\
& -(\beta_3(x_1^2+x_2^2)+\beta_4(x_3^2+x_4^2))x_4,\\
\dot{x}_4
=& \omega_2 x_3+(\beta_3(x_1^2+x_2^2)+\beta_4(x_3^2+x_4^2))x_3\\
& +(\alpha_3(x_1^2+x_2^2)+\alpha_4(x_3^2+x_4^2))x_4
\end{split}
\label{eqn:pd2t}
\end{equation}
up to $O(|x|^3)$ for case~(II),
 where $\alpha_j,\beta_j\in\Rset$, $j=1,\ldots,4$.

Our main results are stated as follows:

\begin{thm}
\label{thm:main1}
Let $n=3$ and suppose that the system~\eqref{eqn:sys} is transformed to \eqref{eqn:pd1}
 up to $O(|x|^2)$.
If $\alpha_1\neq 0$ and one of the following conditions holds,
 then the system~\eqref{eqn:sys}
 is not real-analytically integrable in the Bogoyavlenskij sense near the origin$\,:$
\begin{enumerate}
\setlength{\leftskip}{-1.8em}
\item[\rm(i)]
$\alpha_1\alpha_4>0\,;$
\item[\rm(ii)]
$\alpha_1\alpha_4<0$ and $\alpha_4/\alpha_1\not\in\Qset$.
\end{enumerate}
\end{thm}

\begin{thm}
\label{thm:main2}
Let $n=4$ and suppose that the system~\eqref{eqn:sys} is transformed to \eqref{eqn:pd2}
 up to $O(|x|^3)$.
If $\alpha_1\neq\alpha_3$ and one of the following conditions holds,
 then the system~\eqref{eqn:sys} is not real-analytically integrable
 in the Bogoyavlenskij sense near the origin$\,:$
\begin{enumerate}
\setlength{\leftskip}{-1.6em}
\item[\rm(i)]
$\alpha_2\alpha_3-\alpha_1\alpha_4\neq 0\,;$
\item[\rm(ii)]
$\alpha_2\alpha_4>0\,;$
\item[\rm(iii)]
$\alpha_2\alpha_4<0$ and $\alpha_2/\alpha_4\not\in\Qset$.
\end{enumerate}
\end{thm}

We prove these theorems in Section~4.
The unfoldings of  \eqref{eqn:pd1t} and \eqref{eqn:pd2t},
\begin{equation}
\begin{split}
&\dot{x}_1=\nu x_1-\omega x_2+\alpha_1 x_1x_3-\alpha_2 x_2x_3,\\
&\dot{x}_2=\omega x_1+\nu x_2+\alpha_2 x_1x_3+\alpha_1 x_2x_3,\\
&\dot{x}_3=\mu+\alpha_3(x_1^2+x_2^2)+\alpha_4x_3^2
\end{split}
\label{eqn:fH}
\end{equation}
and
\begin{equation}
\begin{split}
\dot{x}_1
=& -\omega_1 x_2+(\nu+\alpha_1(x_1^2+x_2^2)+\alpha_2(x_3^2+x_4^2))x_1\\
& -(\beta_1(x_1^2+x_2^2)+\beta_2(x_3^2+x_4^2))x_2,\\
\dot{x}_2
=& \omega_1 x_1+(\beta_1(x_1^2+x_2^2)+\beta_2(x_3^2+x_4^2))x_1\\
& +(\nu+\alpha_1(x_1^2+x_2^2)+\alpha_2(x_3^2+x_4^2))x_2,\\
\dot{x}_3
=& -\omega_2 x_4+(\mu+\alpha_3(x_1^2+x_2^2)+\alpha_4(x_3^2+x_4^2))x_3,\\
& -(\beta_3(x_1^2+x_2^2)+\beta_4(x_3^2+x_4^2))x_4,\\
\dot{x}_4
=& \omega_2 x_3+(\beta_3(x_1^2+x_2^2)+\beta_4(x_3^2+x_4^2))x_3\\
& +(\mu+\alpha_3(x_1^2+x_2^2)+\alpha_4(x_3^2+x_4^2))x_4
\end{split}
\label{eqn:dH}
\end{equation}
represent normal forms of fold-Hopf and double-Hopf bifurcations, respectively,
 where $\mu,\nu\in\Rset$ are the control parameters:
At $(\mu,\nu)=(0,0)$, fold (saddle-node) and Hopf bifurcation curves meet for the former
 and two Hopf bifurcation curves for the latter.
Such codimension-two bifurcations are fundamental and interesting phenomena
 in dynamical systems
 and have been studied extensively
 since the seminal papers of Arnold \cite{A72} and Takens \cite{T74}.
See, e.g., \cite{GH83,HI11,K04} for the details.
In \cite{AY20,Y18a},
 the nonintegrability of the normal forms \eqref{eqn:fH} and \eqref{eqn:dH}
 in the Bogoyavlenskij sense were discussed:
They were shown to be meromorphically nonintegrable
 for almost all parameter values of $\alpha_j,\beta_j$, $j=1,2,3,4$,
 near the $x_3$-axis and the $(x_1,x_2)$- or $(x_3,x_4)$-plane, respectively,
 when $(\mu,\nu)\neq(0,0)$,
 while it was not determined whether they are nonintegrable or not when $(\mu,\nu)=(0,0)$.
(A special case of \eqref{eqn:dH} in which $\beta_j=0$, $j=1,2,3,4$,
 was actually considered in \cite{AY20} but their values do  not affect the conclusion,
 as in Theorem~\ref{thm:main2}.)
Our results show that
 not only the normal forms \eqref{eqn:fH} and \eqref{eqn:dH} with $(\mu,\nu)=(0,0)$
 but also the full system \eqref{eqn:sys} is real-analytically nonintegrable
 if the hypotheses of Theorems~\ref{thm:main1} and \ref{thm:main2} hold
 when it is transformed to \eqref{eqn:pd1} or \eqref{eqn:pd2}
 having the $O(|x|^2)$- or $O(|x|^3)$-truncation \eqref{eqn:pd1t} or \eqref{eqn:pd2t}.

We now describe some backgrounds and related work.
For a while, we consider the system \eqref{eqn:sys} 
 in a more general situation in which $n\neq 3,4$ is allowed but $x=0$ is still an equilibrium.

\begin{dfn}[Poincar\'e-Dulac normal form]
Change the coordinates in \eqref{eqn:sys} such that $\D f(0)$ is in Jordan normal form.
The system \eqref{eqn:sys} is called a \emph{Poincar\'e-Dulac (PD) normal form} if $[Sx,f]=0$,
 where $S$ is the semisimple part of $\D f(0)$,
 i.e., $S=\mathrm{diag}\lambda_j$, where $\lambda_j$, $j=1,\ldots,n$, are the eigenvalues of $\D f(0)$.
\end{dfn}

We easily see that the systems \eqref{eqn:pd1} and \eqref{eqn:pd2}
 are PD normal forms for \eqref{eqn:sys} under our assumptions
 although the power series $g_j$, $j=1,2,3$, and $h_j$, $j=1,2,3,4$,  may not be convergent.
Let $\lambda_j$,  $j=1,\ldots,n$, be eigenvalues of $\D f(0)$, and let
\[
\Zset_j^n=\{p=(p_1,\ldots,p_n)\in\Zset^n
 \mid p_j\ge -1,\,p_l\ge 0,\,l\neq j,\,p\neq 0\}
\]
for $j=1,\dots,n$.

\begin{dfn}[Resonance sets and degree]
Let
\[
\R_j=\left\{p\in\Zset_j^n\,\left|\,\sum_{l=1}^n\lambda_jp_j=0\right.\right\},\quad
j=1,\ldots,n,
\]
and let
\[
\R=\bigcup_{j=1}^n\R_j.
\]
We refer to $\R$ as the \emph{resonance set} of \eqref{eqn:sys} and to
\[
\gamma_\r=\dim_\Qset\Span_\Qset\R
\]
as the \emph{resonance degree}  of \eqref{eqn:sys}. 
\end{dfn}
For the PD normal forms \eqref{eqn:pd1} and \eqref{eqn:pd2}
 we easily see that the resonance sets are given by
\[
\R=\Span_\Nset\{(1,0,0),(0,1,1)\}\quad\text{and}\quad
\R=\Span_\Nset\{(1,1,0,0),(0,0,1,1)\},
\]
respectively, and the resonance degrees are $\gamma_\r=2$.
Yamanaka \cite{Y18b} proved the following result for the general case. 

\begin{thm}[Yamanaka]
If the resonance degree $\gamma_\r$ is less than two,
 then the PD normal form is analytically integrable.
Moreover, there exists an $n$-dimensional, analytically nonintegrable PD normal form
 with $\gamma_\r+1$ for $\gamma_\r\ge 2$.
\end{thm}

Similar results for Hamiltonian systems are found in \cite{C12,D84,Y19,Z05}.
The above result does not exclude the analytic nonintegrability
 of \eqref{eqn:pd1} and \eqref{eqn:pd2}.
Actually, he gave a necessary and sufficient condition
 for \eqref{eqn:pd1} to be analytically $(1,2)$-integrable in \cite{Y18b}.
For example, if the system \eqref{eqn:pd1t} is analytically $(1,2)$-integrable,
 then $\alpha_1,\alpha_3,\alpha_4=0$.

On the other hand,
 Zung \cite{Z02} proved the following remarkable result
 on analytically integrable PD normal forms.

\begin{thm}[Zung]
\label{thm:zung}
Let $n\ge 1$ be any integer.
If the system \eqref{eqn:sys} is analytically integrable in the Bogoyavlenskij sense,
 then there exists an analytic change of coordinates
 under which it is transformed to a PD normal form.
\end{thm}

A similar result for Hamiltonian systems was obtained by Zung \cite{Z05}.
Theorem~\ref{thm:zung} also implies that
 the corresponding PD normal form is convergent and analytically integrable
 if the system~\eqref{eqn:sys} is analytically integrable.
Hence, the system \eqref{eqn:sys} is analytically nonintegrable
 if the corresponding PD normal form is divergent or analytically nonintegrable.
So we only have to prove the analytic nonintegrability of \eqref{eqn:pd1} and \eqref{eqn:pd2}
 for the proofs of Theorems~\ref{thm:main1} and \ref{thm:main2}.
In their proofs of the main theorems, we assume
 that the system \eqref{eqn:sys} is analytically integrable 
 and that the power series in \eqref{eqn:pd1} and \eqref{eqn:pd2} are convergent,
 and show that these assumptions yield contradictions.

For the problem on nonintegrability of dynamical systems,
 the Morales-Ramis theory \cite{M99,MR01} and its extension \cite{AZ10,MRS07}
 were developed and have produced numerous remarkable results.
See, e.g., \cite{MP09,M15,MR10} for such examples.
Recently, the author and his coworker also applied the techniques
 and obtained several results on the problem
 for nearly integrable systems in \cite{MY22,Y22a,Y22c},
 for the restricted three-body problems in \cite{Y21a,Y21b}
 and for an epidemic model in \cite{Y22b}.
Here we use a different approach without relying on the techniques.
In particular, a useful relation between first integral and commutative vector fields
 for proving the analytic nonintegrability of planar systems is provided.

The outline of this paper is as follows:
In Section~2 we reduce the nonintegrability of \eqref{eqn:pd1} and \eqref{eqn:pd2}
  to that of simple planar systems.
For this purpose, we use Proposition~2.1 of \cite{AY20},
 which enables us to reduce a special class of systems,
 including \eqref{eqn:pd1} and \eqref{eqn:pd2}, to planar systems,
 along with a simple but clever trick.
In Section~3 we provide the useful relation 
 on first integrals and commutative vector fields.
In Section~4 we prove the main theorems
 using the results of Sections~2 and 3.
Finally, to demonstrate  our results,
 we give two examples for the R\"ossler system \cite{CNV20,L14,K04,MBKPS20,ZY20}
 and coupled van der Pol oscillators \cite{CR88,IOSK04,LRS03,PFG14,RH80,SR82,SR00}
 in Section~5.
 

\section{Reduction to Simple Planar Systems}

In this section
 we reduce the nonintegrability of \eqref{eqn:pd1} and \eqref{eqn:pd2}
 to that of simple planar systems.

Using the change of coordinate $(x_1,x_2)=(r\cos\theta,r\sin\theta)$,
 we transform \eqref{eqn:pd1t} to
\begin{equation}
\dot r=g_1(r^2,x_3)r,\quad
\dot x_3=g_3(r^2,x_3),\quad
\dot\theta=\omega+g_2(r^2,x_3),
\label{eqn:p1}
\end{equation}
of which the $(r_1,r_2)$-components are independent of $\theta$.
Using the change of coordinates $(x_1,x_2)=(r_1\cos\theta_1,r_1\sin\theta_1)$
 and $(x_3,x_4)=(r_2\cos\theta_2,r_2\sin\theta_2)$,
 we also transform \eqref{eqn:pd2t} to
\begin{equation}
\begin{split}
&
\dot{r}_1=h_1(r_1^2,r_2^2)r_1,\quad
\dot{r}_2=h_3(r_1^2,r_2^2)r_2,\\
&
\dot{\theta}_1=\omega_1+h_2(r_1^2,r_2^2),\quad
\dot{\theta}_2=\omega_2+h_4(r_1^2,r_2^2)
\end{split}
\label{eqn:p2}
\end{equation}
of which the  $(r_1,r_2)$-components are independent of $\theta_1$ and $\theta_2$.
So we expect that one can reduce the nonintegrability of \eqref{eqn:pd1t} and \eqref{eqn:pd2t}
 to that of the $(r,x_3)$-components of \eqref{eqn:p1},
\begin{equation}
\dot r=g_1(r^2,x_3)r,\quad
\dot x_3=g_3(r^2,x_3),
\label{eqn:p10}
\end{equation}
and the $(r_1,r_2)$-components of \eqref{eqn:p2},
\begin{equation}
\dot{r}_1=h_1(r_1^2,r_2^2)r_1,\quad
\dot{r}_2=h_3(r_1^2,r_2^2)r_2,
\label{eqn:p20}
\end{equation}
respectively.
This is true in a more general situation as follows.

Let $m>0$ be an integer and consider $m+2$-dimensional systems of the form
\begin{equation}
\dot{x}=f_x(x,y),\quad
\dot{y}=f_y(x,y),\quad
(x,y)\in D,
\label{eqn:fg}
\end{equation}
where $D\subset\Cset^2\times\Cset^m$ is a region containing
 the $m$-dimensional $y$-plane $\{(0,y)\in\Cset^2\times\Cset^m\mid y\in\Cset^m\}$,
 and $f_x:D\to\Cset^2$ and $f_y:D\to\Cset^m$ are analytic.
Assume that by the change of coordinates
 $x=(x_1,x_2)=(r\cos\theta,r\sin\theta)$,
 Eq.~\eqref{eqn:fg} is transformed to
\begin{equation}
\dot{r}=R(r,y),\quad
\dot{y}=\tilde{f}_y(r,y),\quad
\dot{\theta}=\Theta(r,y),\quad
(r,y,\theta)\in\tilde{D}\times\Cset,
\label{eqn:Rg0}
\end{equation}
where $\tilde{D}\subset\Cset\times\Cset^m$ is a region
 containing the $m$-dimensional $y$-plane,
 and $R:\tilde{D}\to\Cset$, $\tilde{f}_y:\tilde{D}\to\Cset^m$
 and $\Theta:\tilde{D}\to\Rset$ are analytic.
Note that $\tilde{f}_y(r,y)=f_y(r\cos\theta,r\sin\theta,y)$.
We are especially interested in the $(r,y)$-components of \eqref{eqn:Rg0},
\begin{equation}
\dot{r}=R(r,y),\quad
\dot{y}=\tilde{f}_y(r,y),
\label{eqn:Rg}
\end{equation}
which are independent of $\theta$.
In this situation we have the following proposition.

\begin{prop}
\label{prop:2a}\
\begin{enumerate}
\setlength{\leftskip}{-1.8em}
\item[\rm(i)]
Suppose that Eq.~\eqref{eqn:fg} has a meromorphic first integral $F(x_1,x_2,y)$
 near $(x_1,x_2)\linebreak=(0,0)$,
 and let $\tilde{F}(r,\theta,y)=F(r\cos\theta,r\sin\theta,y)$.
If $\tilde{f}_{yj}(0,y)\neq 0$ for almost all $y\in\tilde{D}$ for some $j=1,\ldots,m$, then
\[
G(r,y)=\tilde{F}(r,\tilde{\theta}_j(y_j),y)
\]
is a meromorphic first integral of \eqref{eqn:Rg} near $r=0$,
 where $y_j$ and $\tilde{f}_{yj}(r,y)$ are the $j$-th components of $y$ and $\tilde{f}_y(r,y)$, respectively,
 and $\tilde{\theta}_j(y_j)$ represents the $\theta$-component of a solution to
\[
\frac{\d r}{\d y_j}=\frac{R(r,y)}{\tilde{f}_{yj}(r,y)},\quad
\frac{\d y_\ell}{\d y_j}=\frac{\tilde{f}_{y\ell}(r,y)}{\tilde{f}_{yj}(r,y)},\quad
\frac{\d\theta}{\d y_j}=\frac{\Theta(r,y)}{\tilde{f}_{yj}(r,y)},\quad\ell\neq j.
\]
\item[\rm(ii)]
Suppose that Eq.~\eqref{eqn:fg} has a  meromorphic commutative vector field
\[
v(x_1,x_2,y):=
\begin{pmatrix}
v_1(x_1,x_2,y)\\
v_2(x_1,x_2,y)\\
v_y(x_1,x_2,y)
\end{pmatrix}
\]
with $v_1,v_2:D\to\Cset$ and $v_y:D\to\Cset^m$ near $(x_1,x_2)=(0,0)$.
If $\Theta(0,y)\neq 0$ for almost all $y\in\tilde{D}$, then
\begin{align*}
&
\begin{pmatrix}
\tilde{v}_r(r,\theta,y)\\
\tilde{v}_y(r,\theta,y)
\end{pmatrix}\notag\\
&=
\begin{pmatrix}
v_1(r\cos\theta,r\sin\theta,y)\cos\theta+v_2(r\cos\theta,r\sin\theta,y)\sin\theta\\
v_y(r\cos\theta,r\sin\theta,y)
\end{pmatrix}
\end{align*}
is independent of $\theta$
 and it is a meromorphic commutative vector field of \eqref{eqn:Rg} near $r=0$.
\end{enumerate}
\end{prop}

See Proposition~2.1 of \cite{AY20} for the proof.
Using Proposition~\ref{prop:2a} for \eqref{eqn:pd1t} and \eqref{eqn:pd2t}
 (once for the former and twice for the latter),
 we immediately obtain the following propositions.

\begin{prop}
\label{prop:p1}
If the complexification of \eqref{eqn:sys} in case~{\rm(I)} is meromorphically integrable
 near $(x_1,x_2)=(0,0)$, then so is the system~\eqref{eqn:p10} near $r=0$.
\end{prop}

\begin{prop}
\label{prop:p2}
If the complexification of \eqref{eqn:sys} in case~{\rm(II)} is meromorphically integrable
 near $(x_1,x_2)=(0,0)$ and near $(x_3,x_4)=(0,0)$,
 then so is the system~\eqref{eqn:p20} near $r_1=0$ and near $r_2=0$, respectively.
\end{prop}

We turn to systems of the general form \eqref{eqn:sys} with $n\ge 2$
 but $f(0)=0$, $\D f(0)=0,\ldots,\D^{k-1}f(0)=0$ and $\D^k f(0)\neq 0$ for some $k\in\Nset$.
Since $f(x)$ is analytic near $x=0$, we have
\begin{equation}
f(x)=\sum_{j=k}^\infty f_j(x),
\label{eqn:f}
\end{equation}
where the elements of $f_j(x)$ are $j$th-order homogeneous polynomials of $x$.
Letting $x=\epsilon y$ and changing the time variable as $t\to\epsilon^k t$,
 we rewrite \eqref{eqn:sys} as
\[
\dot{y}=\sum_{j=0}^\infty\epsilon^j f_{j+k}(y).
\]
We prove the following result.

\begin{thm}
\label{thm:2b}
Suppose that $f(x)$ has the form \eqref{eqn:f} for some $k\in\Nset$.
If the system \eqref{eqn:sys} is analytically integrable in the Bogoyavlenskij sense,
 then so is the truncated system
\begin{equation}
\dot{y}=f_k(y).
\label{eqn:yk}
\end{equation}
\end{thm}

\begin{proof}
Let $F(x)$ be an analytic first integral of \eqref{eqn:sys} near $x=0$,
 and let
\begin{equation}
F(x)=\sum_{j=\ell}^\infty F_j(x)
\label{eqn:thm2b}
\end{equation}
for some $\ell\in\Nset$,
 where $F_j(x)$ is a $j$th-order homogeneous polynomial of $x$.
Here we have assumed that $F(0)\equiv 0$ without loss of generality.
Then we have
\[
\D F(\epsilon y)^\T f(\epsilon y)
 =\sum_{j=0}^\infty\sum_{l=0}^\infty\epsilon^{k+\ell+j+l}\D F_{\ell+j}(y)^\T f_{k+l}(y)\equiv 0,
\]
in particular,
\[
\D F_\ell(y)^\T f_k(y)\equiv 0,
\]
which means that $F_\ell(y)$ is an analytic first integral of the truncated system \eqref{eqn:yk}.

On the other hand,
 let $v(x)$ be an analytic commutative vector field of \eqref{eqn:sys} near $x=0$.
Let
\[
v(\epsilon y)=\sum_{j=0}^\infty \epsilon^j v_j(y),
\]
where the elements of $v_j(x)$ are $j$th-order homogeneous polynomials of $x$,
 and assume that $v_j(y)\equiv 0$, $j=0,\ldots,\ell-1$, and $v_\ell(y)\not\equiv 0$
 for some $\ell\in\Nset$.
Then we have
\begin{align*}
[v,f](\epsilon y)=&\D f(\epsilon y)v(\epsilon y)-\D v(\epsilon y)f(\epsilon y)\\
=&\sum_{j=0}^\infty\sum_{l=0}^\infty\epsilon^{k+\ell+j+l}
 \D f_{k+l}(y)v_{\ell+j}(y)-\D v_{\ell+j}(y)f_{k+l}(y)\equiv 0,
\end{align*}
in particular,
\[
\D f_k(y)v_\ell(y)-\D v_\ell(y)f_k(y)\equiv 0,
\]
which means that $v_\ell(y)$ is an analytic commutative vector field
 of the truncated system \eqref{eqn:yk}.
 
Suppose that the system \eqref{eqn:sys} is analytically integrable.
Then we can choose the analytic first integrals (resp. commutative vector fields)
 such that their leading terms are linearly independent almost everywhere in a neighborhood of $x=0$,
 by taking their linear combinations if necessary.
Actually, for instance,
 if $F(x)$ and $G(x)$ are linearly independent first integrals with \eqref{eqn:thm2b} and
\[
G(x)=\sum_{j=\ell}^\infty G_j(x),
\]
where $G_j(x)$ is a $j$th-order homogeneous polynomials of $x$,
 and for some $m>\ell$,
\[
\sum_{j=\ell}^{m-1}(c_1\D F_j(x)+c_2\D G_j(x))=0
\]
for some $(c_1,c_2)\neq(0,0)$ but
\[
\sum_{j=\ell}^{m}(\tilde{c}_1\D F_j(x)+\tilde{c}_2\D G_j(x))\neq 0
\]
for any $(\tilde{c}_1,\tilde{c}_2)\neq(0,0)$, then one may take $F(x)$ and
\[
\tilde{G}(x)=c_1 F(x)+c_2 G(x),
\]
for which the leading term is
\[
c_1\D F_m(x)+c_2\D G_m(x),
\]
as two new linearly independent first integrals.
So we show that the leading terms of first integrals and commutative vector fields
 satisfy conditions (i) and (ii) of Definition~\ref{dfn:1a} for \eqref{eqn:yk},
 along with the above observations.
Thus, we obtain the desired result.
\end{proof}

\begin{rmk}\
\begin{enumerate}
\setlength{\leftskip}{-1.8em}
\item[\rm(i)]
In contrast to Theorem~$\ref{thm:2b}$, the truncated system
\[
\dot{y}=\sum_{j=k}^{k+m}f_j(x)
\]
with $m\ge 1$ may not be analytically integrable in general
 even if the full system \eqref{eqn:f} is analytically integrable.
Actually, Yoshida {\rm\cite{Y88}} showed that
 the truncation of the three-particle Toda lattice {\rm\cite{T81}} with $k=1$ 
 is analytically nonintegrable at any order $m\ge 1$
 although the Toda lattice is analytically integrable as well known {\rm\cite{F74,H74}}.
\item[\rm(ii)]
An argument similar to that in the above proof was used
 for a three-degree-of-freedom Hamiltonian system in Section~$2$ of {\rm\cite{S18}}.
\end{enumerate}
\end{rmk}

Applying Theorem~\ref{thm:2b} to \eqref{eqn:p10} and \eqref{eqn:p20}
 and using Propositions~\ref{prop:p1} and \ref{prop:p2},
 we obtain the following.

\begin{prop}
\label{prop:p1t}
If the complexification of \eqref{eqn:sys} in case {\rm(I)} is analytically integrable
 near the origin $x=0$, then so is the truncated system
\begin{equation}
\dot r=\alpha_1 rx_3,\quad
\dot x_3=\alpha_3r^2+\alpha_4x_3^2
\label{eqn:p1t}
\end{equation}
near $(r,x_3)=(0,0)$.
\end{prop}

\begin{prop}
\label{prop:p2t}
If the complexification of \eqref{eqn:sys} in case {\rm(II)} is analytically integrable
 near the origin $x=0$, then  so is the truncated system
\begin{equation}
\dot{r}_1
=(\alpha_1 r_1^2+\alpha_2r_2^2)r_1,\quad
\dot{r}_2
= (\alpha_3 r_1^2+\alpha_4r_2^2)r_2
\label{eqn:p2t}
\end{equation}
near $(r_1,r_2)=(0,0)$.
\end{prop}

\begin{rmk}
\label{rmk:2a}
If the system~\eqref{eqn:sys} is real-analytically integrable near $x=0$,
 then its complexification is also analytically integrable near $x=0$.
So we only have to prove that
 the complexifications of \eqref{eqn:p1t} and \eqref{eqn:p2t}
 are analytically nonintegrable near $x=0$
 for the proofs of Theorems~$\ref{thm:main1}$ and $\ref{thm:main2}$.
\end{rmk}


\section{
Planar Vector Fields}

In this section we give a useful relation between first integrals and commutative vector fields
 for proving the analytic nonintegrability
 of such planar systems as \eqref{eqn:p1t} and \eqref{eqn:p2t}.

Consider planar vector fields of the form
\begin{equation}
\dot{z}=p(z),\quad
z\in\Cset^2,
\label{eqn:psys}
\end{equation}
where $p(z)$ is analytic in $z$.
We prove the following.

\begin{prop}
\label{prop:3a}
Let $D\subset\Cset$ be a region
 that is covered by nonconstant solutions to \eqref{eqn:psys} almost everywhere.
Suppose that the system~\eqref{eqn:psys} has a first integral $Q(x)$
 and commutative vector field $q(x)$ in $D$.
Let
\[
\Delta(x)=\det(p(x),q(x))
= p_1(x)q_2(x)-p_2(x)q_1(x),
\]
where $q_j(x)$ and $p_j(x)$ are the $j$th-elements of $q(x)$ and $p(x)$, respectively.
Then there exists a function $\chi:\Cset\to\Cset$ such that
\begin{equation}
\Delta(x)\D Q(x)=\chi(Q(x))
\begin{pmatrix}
q_2(x)\\
-q_1(x)
\end{pmatrix}.
\label{eqn:prop3a}
\end{equation}
\end{prop}

\begin{proof}
Let $z=\varphi(t)$ be a nonconstant particular solution to \eqref{eqn:psys}.
We begin with the following lemmas.

\begin{lem}
\label{lem:3a}
If the planar system \eqref{eqn:psys} has a commutative vector field $q(z)$
 $($resp. a first integral $Q(z))$,
 then $\xi=q(\varphi(t))$ is a solution to the variational equation $($\!VE$)$
 of \eqref{eqn:psys} along $\varphi(t)$,
\begin{equation}
\dot{\xi}=\D p(\varphi(t))\xi
\label{eqn:ve}
\end{equation}
$($resp. then $\eta=\D Q(\varphi(t))$ is a solution to the adjoint variational equation $($\!AVE$)$
 of \eqref{eqn:psys} along $\varphi(t)$,
\begin{equation}
\dot{\eta}=-\D p(\varphi(t))^\T\eta\ ).
\label{eqn:ave}
\end{equation}
\end{lem}

\begin{proof}
Let $q(z)$ be a commutative vector field of \eqref{eqn:psys}.
Then
\[
\D q(z)p(z)-\D p(z)q(z)=0,
\]
so that
\[
\frac{\d}{\d t}q(\varphi(t))=\D p(\varphi(t))q(\varphi(t)).
\]
Hence, $\xi=q(\varphi(t))$ is a solution to \eqref{eqn:ve}.

On the other hand, let $Q(z)$ be a first integral of \eqref{eqn:psys}.
Then
\[
p(z)^\T\D Q(z)=0,
\]
so that
\[
\D(p(z)^\T\D Q(z))=\D p(z)^\T\D Q(z)+\D^2 Q(z)p(z)=0.
\]
Hence,
\begin{align*}
&
\frac{\d}{\d t}\D Q(\varphi(t))
=\D^2Q(\varphi(t))p(\varphi(t))
=-\D p(\varphi(t))^\T\D Q(\varphi(t)),
\end{align*}
which means that $\eta=\D Q(\varphi(t))$ is a solution to \eqref{eqn:ave}.
\end{proof}

\begin{lem}
\label{lem:3b}
Let $\Phi(t)$ and $\Psi(t)$ be fundamental matrices
 to the VE \eqref{eqn:ve} and AVE \eqref{eqn:ave}, respectively.
Then
\[
\Phi(t)^\T\Psi(t)=\text{const.}
\]
\end{lem}

\begin{proof}
We easily compute
\begin{align*}
\frac{\d}{\d t}(\Phi(t)^\T\Psi(t))
=&\dot{\Phi}(t)^\T\Psi(t)+\Phi(t)^\T\dot{\Psi}(t)\\
=&\Phi(t)^\T\D p(\varphi(t))^\T\Psi(t)-\Phi(t)^\T\D p(\varphi(t))^\T\Psi(t)=0,
\end{align*}
which yields the desired result.
\end{proof}

We return to the proof of Proposition~\ref{prop:3a}.
By Lemma~\ref{lem:3a} $\xi=q(\varphi(t))$ and $\eta=\D Q(\varphi(t))$
 are solutions to the VE \eqref{eqn:ve} and AVE \eqref{eqn:ave}, respectively.
Let $\tilde{\eta}(t)$ be another linearly independent solution to \eqref{eqn:ave}.
Noting that $\xi=p(\varphi(t))=\dot{\varphi}(t)$
 is another linearly independent solution to \eqref{eqn:ve},
 we see by Lemma~\ref{lem:3b} that 
\begin{equation}
\begin{pmatrix}
p(\varphi(t))^\T\\
q(\varphi(t))^\T
\end{pmatrix}
(\D Q(\varphi(t))\ \tilde{\eta}(t))=\text{const.},
\label{eqn:thm3a1}
\end{equation}
so that
\begin{equation}
\begin{pmatrix}
p(x)^\T\\
q(x)^\T
\end{pmatrix}
\D Q(x)=
\begin{pmatrix}
C(Q(x)) \\
0
\end{pmatrix}
\label{eqn:thm3a2}
\end{equation}
holds almost everywhere in $D$,
 where $C(Q(x))\neq 0$ is a constant only depending on the value of $Q(x)$,
 since Eq.~\eqref{eqn:thm3a1} holds at any point $x=\varphi(t)$ on the nonconstant solution
 for the same constant matrix in its right hand side.
Note that $\D Q(x)^\T q(x)=q(x)^\T\D Q(x)=0$ since $Q(x)$ is a first integral of \eqref{eqn:psys}.
The matrix
\[
\begin{pmatrix}
p(x)^\T\\
q(x)^\T
\end{pmatrix}
\]
is nonsigular and its inverse matrix is given by
\[
\frac{1}{\Delta(x)}
\begin{pmatrix}
q_2(x) & -p_2(x)\\
-q_1(x) & p_1(x)
\end{pmatrix}.
\]
From \eqref{eqn:thm3a2} we obtain
\[
\D Q(x)=\frac{1}{\Delta(x)}
\begin{pmatrix}
C(Q(x))q_2(x)\\
-C(Q(x))q_1(x)
\end{pmatrix},
\]
which yields \eqref{eqn:prop3a} with $\chi(Q)=C(Q)$.
\end{proof}


\section{Proofs of the Main Theorems}

We are now in a position to prove Theorems~\ref{thm:main1} and \ref{thm:main2}.

\subsection{Proof of Theorem~\ref{thm:main1}}

By Proposition~\ref{prop:p1t} and Remark~\ref{rmk:2a},
 Theorem~\ref{thm:main1} immediately follows from the following proposition.

\begin{prop}
\label{prop:4a}
If $\alpha_1\neq 0$ and one of the following conditions holds,
 then the truncated system \eqref{eqn:p1t} is analytically nonintegrable near $(r,x_3)=(0,0):$
\begin{enumerate}
\setlength{\leftskip}{-1.8em}
\item[\rm(i)]
$\alpha_1\alpha_4>0\,;$
\item[\rm(ii)]
$\alpha_1\alpha_4<0$ and $\alpha_4/\alpha_1\not\in\Qset$.
\end{enumerate}
\end{prop}

\begin{proof}
Assume that $\alpha_1\neq 0$ and condition~(i) or (ii) holds.
We easily see that the system~\eqref{eqn:p1t} has no constant solution except for $(r,x_3)=(0,0)$
 and a first integral
\[
Q(r,x_3)=r^{-2\alpha_4/\alpha_1}(\alpha_3r^2+(\alpha_4-\alpha_1)x_3^2),
\]
for which
\begin{align*}
&
\D_r Q(r,x_3)=\frac{2(\alpha_1-\alpha_4)}{\alpha_1}r^{-2\alpha_4/\alpha_1-1}
 (\alpha_3r^2+\alpha_4x_3^2),\\
&
\D_{x_3}Q(r,x_3)=-2(\alpha_1-\alpha_4)r^{-2\alpha_4/\alpha_1}x_3.
\end{align*}
Obviously, $Q(r,x_3)$ is not analytic.
Moreover, when $\alpha_1\neq 0$,
 if there exists an analytic first integral,
 then $\alpha_1\alpha_4\le 0$ and $\alpha_4/\alpha_1\in\Qset$.
Hence, no analytic first integral exists.

Assume that the system~\eqref{eqn:p1t} has a commutative vector field $q(r,x_3)$.
Let $p(r,x_3)$ denote the vector field of \eqref{eqn:p1t}.
We compute
\[
\frac{Q(r,x_3)p_2(r,x_3)}{\D_r Q(r,x_3)}
=-\frac{Q(r,x_3)p_1(r,x_3)}{\D_{x_3}Q(r,x_3)}
=\frac{\alpha_1r (\alpha_3r^2+(\alpha_4-\alpha_1)x_3^2)}{2(\alpha_1-\alpha_4)},
\]
so that by Proposition~\ref{prop:3a}
\begin{equation}
\Delta(r,x_3)=Cr (\alpha_3r^2+(\alpha_4-\alpha_1)x_3^2),
\label{eqn:4a}
\end{equation}
where $C\neq 0$ is some constant, since $Q(r,x_3)$ is not analytic.
We write the Taylor expansion of $q_j(r,x_3)$ around the origin as
\[
q_j(r,x_3)=\sum_{k,l=1}^{\infty}q_{jkl}r^kx_3^l,
\]
where $q_{jkl}\in\Cset$, $k,l=1,2$,  are constants, for $j=1,2$.
Substituting them into \eqref{eqn:4a} and solving the resulting equation about $p_{jkl}$, we obtain
\[
q(r,x_3)=C
\begin{pmatrix}
r\\
x_3
\end{pmatrix}+O(|r|^2+|x_3|^2).
\]
So we have
\[
\D p(r,x_3)q(r,x_3)-\D q(r,x_3)p(x)=C
\begin{pmatrix}
\alpha_1rx_3\\
\alpha_3r^2+\alpha_4x_3^2
\end{pmatrix}+O(|r|^3+|x_3|^3),
\]
which means that $q(r,x_3)$ is not a commutative vector field.
Thus, we obtain the desired result.
\end{proof}

\subsection{Proof of Theorem~\ref{thm:main2}}

As in Section~4.1, by Proposition~\ref{prop:p2t} and Remark~\ref{rmk:2a},
 Theorem~\ref{thm:main2} immediately follows from the following proposition.

\begin{prop}
\label{prop:4b}
If $\alpha_1\neq\alpha_3$ and one of the following conditions holds,
 then the truncated system \eqref{eqn:p2t} is analytically nonintegrable near $(r_1,r_2)=(0,0):$
\begin{enumerate}
\setlength{\leftskip}{-1.6em}
\item[\rm(i)]
$\alpha_2\alpha_3-\alpha_1\alpha_4\neq 0\,;$
\item[\rm(ii)]
$\alpha_2\alpha_4>0\,;$
\item[\rm(iii)]
$\alpha_2\alpha_4<0$ and $\alpha_2/\alpha_4\not\in\Qset$.
\end{enumerate}
\end{prop}

\begin{proof}
Assume that $\alpha_1\neq\alpha_3$ and condition~(i), (ii) or (iii) holds.
We easily see that the system~\eqref{eqn:p2t} has no constant solution except for $(r,x_3)=(0,0)$
 and a first integral
\[
Q(r_1,r_2)
=\left(\frac{r_1^{\alpha_4}}{r_2^{\alpha_2}}\right)^{2(\alpha_1-\alpha_3)}
 \left(\frac{(\alpha_1-\alpha_3)r_1^2}{r_2^2}
 +\alpha_2-\alpha_4\right)^{\alpha_2\alpha_3-\alpha_1\alpha_4},
\]
for which
\begin{align*}
\D_{r_1}Q(r_1,r_2)
=&\frac{2(\alpha_1-\alpha_3)(\alpha_2-\alpha_4)(\alpha_3r_1^2+\alpha_4r_2^2)}
 {r_1((\alpha_1-\alpha_3)r_1^2+(\alpha_2-\alpha_4)r_2^2)}Q(r_1,r_2),\\
\D_{r_2}Q(r_1,r_2)
=&-\frac{2(\alpha_1-\alpha_3)(\alpha_2-\alpha_4)(\alpha_1r_1^2+\alpha_2r_2^2)}
 {r_2((\alpha_1-\alpha_3)r_1^2+(\alpha_2-\alpha_4)r_2^2)}Q(r_1,r_2).
\end{align*} 
Obviously, $Q(r_1,r_2)$ is not analytic.
Moreover, if there exists an analytic first integral,
 then $\alpha_2\alpha_3-\alpha_1\alpha_4=0$ and one of the following conditions holds:
\begin{itemize}
\setlength{\leftskip}{-2.5em}
\item
$\alpha_2\alpha_4=0$;
\item
$\alpha_2\alpha_4<0$ and $\alpha_2/\alpha_4\in\Qset$.
\end{itemize}
Hence, no analytic first integral exists under our assumption.

Assume that the system~\eqref{eqn:p2t} has a commutative vector field $q(r_1,r_2)$.
Let $p(r_1,r_2)$ denote the vector field of \eqref{eqn:p2t}.
We have
\[
\frac{Q(r_1,r_2)p_2(r_1,r_2)}{\D_{r_1} Q(r_1,r_2)}
=-\frac{Q(r_1,r_2)p_1(r_1,r_2)}{\D_{r_2}Q(r_1,r_2)}
=\tfrac{1}{2}r_1r_2\left(\frac{r_1^2}{\alpha_2-\alpha_4}+\frac{r_1^2}{\alpha_1-\alpha_3}\right),
\]
so that by Proposition~\ref{prop:3a}
\begin{equation}
\Delta(r_1,r_2)=Cr_1r_2\left(\frac{r_1^2}{\alpha_2-\alpha_4}+\frac{r_2^2}{\alpha_1-\alpha_3}\right),
\label{eqn:4b}
\end{equation}
where $C\neq 0$ is some constant, since $Q(r_1,r_2)$ is not analytic.
We write the Taylor expansion of $q_j(r_1,r_2)$ around the origin as
\[
q_j(r_1,r_2)=\sum_{k,l=1}^{\infty}q_{jkl}r_1^kr_2^l,
\]
where $q_{jkl}\in\Cset$, $k,l=1,2$, are constants, for $j=1,2$.
Substituting them into \eqref{eqn:4a} and solving the resulting equation about $q_{jkl}$,
 we obtain
\[
q(r_1,r_2)=-\frac{C}{(\alpha_1-\alpha_3)(\alpha_2-\alpha_4)}
\begin{pmatrix}
r_1\\
r_2
\end{pmatrix}+O(|r|_1^2+|r_2|^2).
\]
So we have
\begin{align*}
&
\D p(r_1,r_2)q(r_1,r_2)-\D q(r_1,r_2)p(x)\\
&
=-\frac{2C}{(\alpha_1-\alpha_3)(\alpha_2-\alpha_4)}
\begin{pmatrix}
(\alpha_1r_1^2+\alpha_2r_2^2)r_1\\
(\alpha_3r_1^2+\alpha_4r_2^2)r_2
\end{pmatrix}+O\left((|r|^4+|x_3|^4\right),
\end{align*}
which means that $q(r_1,r_2)$ is not a commutative vector field.
Thus, we obtain the desired result.
\end{proof}


\section{Examples}

As stated in Section~1,
Theorems~\ref{thm:main1} and \ref{thm:main2} imply that three- or four-dimensional systems
 exhibiting fold-Hopf and double-Hopf bifurcations are analytically nonintegrable
 under the weak conditions.
In this section we give two such examples.
 
\subsection{R\"ossler system}

We first consider the three-dimensional system
\begin{equation}
\dot{x}_1=-(x_2+x_3),\quad
\dot{x}_2=x_1+ax_2,\quad
\dot{x}_3=bx_1+x_3(x_1-c),
\label{eqn:rsys}
\end{equation}
where $a,b,c$ are constants.
The R\"ossler system
\begin{equation}
\dot{x}_1=-(x_2+x_3),\quad
\dot{x}_2=x_1+a_0x_2,\quad
\dot{x}_3=b_0+x_3(x_1-c_0),
\label{eqn:rsys0}
\end{equation}
which was originally proposed by R\"ossler \cite{R76}
 and has been extensively studied,
e.g., in \cite{BBS09,GZ21,GZ22,LDM95,LV07,LZ02,R16,WZ09,Z97,Z04},
 is transformed to \eqref{eqn:rsys} with
\[
a=a_0,\quad
b=\frac{c_0\pm\sqrt{c_0^2-4a_0b_0}}{2a_0},\quad
c=\tfrac{1}{2}\left(c_0\pm\sqrt{c_0^2-4a_0b_0}\right),
\]
by a change of coordinates
\[
x\mapsto x-\left(x_{10},-\frac{x_{10}}{a_0},\frac{x_{10}}{a_0}\right),\quad
x_{10}=\tfrac{1}{2}\left(c_0\mp\sqrt{c_0^2-4a_0b_0}\right)
\]
if $c_0^2>4a_0b_0$ and $a_0\neq 0$, where the upper or lower sign is taken simultaneously,
 and with
\[
a=0,\quad
b=\frac{b_0}{c_0},\quad
c=c_0,
\]
by a change of coordinates
\[
x\mapsto x-\left(0,-\frac{b_0}{c_0},\frac{b_0}{c_0}\right)
\]
if $a_0=0$ and $c_0\neq 0$. 
The system \eqref{eqn:rsys} has also been referred to as the R\"ossler system in some references.
Periodic orbits, invariant tori, chaos  and fold-Hopf bifurcations in \eqref{eqn:rsys}
 were studied in \cite{CNV20,L14,MBKPS20,ZY20}.
Moreover, the $(1,2)$- or $(2,1)$-integrability of \eqref{eqn:rsys} and/or \eqref{eqn:rsys0}
 was discussed in \cite{LV07,LZ02,Z04} and the following results
 were obtained:
\begin{itemize}
\setlength{\leftskip}{-1.8em}
\item[\rm(i)]
The systems~\eqref{eqn:rsys} and \eqref{eqn:rsys0}
 are analytically $(1,2)$-integrable when $a,b,c=0$ and $a_0,b_0,c_0=0$, respectively;
\item[\rm(ii)]
The system~\eqref{eqn:rsys0} is neither analytically $(1,2)$- nor $(2,1)$-integrable near $x=0$
 when $a_0\neq 0$.
\end{itemize}
The second statement above means that
 the system~\eqref{eqn:rsys} is  neither analytically $(1,2)$- nor $(2,1)$-integrable
 near any point on the line
\[
\{(x_1,x_2,x_3)\in\Rset^3\mid x_2=-ax_1,x_3=ax_1\},
\]
especially near the origin when $a\neq 0$ and $c=2ab\neq 0$.
However, we cannot deny from the statement
 that the systems~\eqref{eqn:rsys} and \eqref{eqn:rsys0}
 may be analytically $(3,0)$-integrable near the origin even when $a_0\neq 0$.

When $b=1$ and $c=a\in(-\sqrt{2},\sqrt{2})$,
 the system \eqref{eqn:rsys} satisfies condition~(I)
 with $\omega=\sqrt{2-a^2}$.
We compute the coefficients in \eqref{eqn:pd1t} as
\begin{equation}
\alpha_1=-\frac{a^3}{2\omega^2},\quad
\alpha_2=\frac{a^2+1}{2\omega},\quad
\alpha_3=\frac{2a}{\omega^2},\quad
\alpha_4=\frac{a}{\omega^2}.
\label{eqn:5a}
\end{equation}
See Appendix~A.1 for the derivation of \eqref{eqn:5a}.
Applying Theorem~\ref{thm:main1}, we obtain the following.

\begin{prop}
\label{prop:5a}
When $b=1$, $c=a\in(-\sqrt{2},\sqrt{2})$ and $a^2\not\in\Qset$,
 the R\"ossler system \eqref{eqn:rsys} is not real-analytically integrable near the origin.
\end{prop}

\subsection{Coupled van der Pol oscillators}
We turn to the second example,
 the coupled van der Pol oscillators
\begin{align*}
&
\ddot{u}_1-(\delta_1-a_1u_1^2)\dot{u}_1+u_1=b_1u_2,\\
&
\ddot{u}_2-(\delta_2-a_2u_2^2)\dot{u}_2+cu_2=b_2u_1,
\end{align*}
or as a first-order system  
\begin{equation}
\begin{split}
&
\dot{x}_1=x_2,\quad
\dot{x}_2=-x_1+(\delta_1-a_1x_1^2)x_2+b_1x_3,\\
&
\dot{x}_3=x_4,\quad
\dot{x}_4=-cx_3+(\delta_2-a_2x_3^2)x_4+b_2x_1,
\end{split}
\label{eqn:cvdp}
\end{equation}
where $x=(u_1,\dot{u}_1,u_2,\dot{u}_2)$,
  and $\delta_j,a_j,b_j,c\in\Rset$, $j=1,2$, and $c>0$ are constants.
The coupled van der Pol oscillators such as \eqref{eqn:cvdp}
 have attracted much attention in the field of dynamical systems
 and for instance, their dynamics and bifurcations
 with $a_j=\delta_j>0$, $j=1,2$, or $a_j=\delta_1>0$
 were studied in \cite{CR88,IOSK04,LRS03,PFG14,RH80,SR82,SR00}.
For simplicity we assume that $a_j,b_j>0$, $j=1,2$, and $c>1$.

When $\delta_j=0$, $j=1,2$, and $b_1b_2<c$,
 the system~\eqref{eqn:cvdp} satisfies condition~(II) with
\begin{align*}
\omega_1=\sqrt{\frac{(c+1)-\sqrt{(c-1)^2+4b_1b_2}}{2}},\quad
\omega_2=\sqrt{\frac{(c+1)+\sqrt{(c-1)^2+4b_1b_2}}{2}}
\end{align*}
if $\omega_1/\omega_2\not\in\Qset$.
We easily see see that $\omega_1<1<\omega_2$ and that
\begin{equation}
\omega_1^2+\omega_2^2=c+1,\quad
\omega_1^2\omega_2^2=c-b_1b_2.
\label{eqn:5brel}
\end{equation}
We compute the coefficients in \eqref{eqn:pd2t} as
\begin{equation}
\begin{split}
&
\alpha_1=\frac{a_1b_1(\omega_2^2-1)^2+a_2b_2(\omega_1^2-1)^2}
 {2b_2\omega_1^2(\omega_1^2-1)(\omega_2^2-\omega_1^2)},\quad
\alpha_2=\frac{(\omega_1^2-1)(a_1b_1+a_2b_2)}{b_2\omega_2^2(\omega_2^2-\omega_1^2)},\\
&
\alpha_3=-\frac{(\omega_2^2-1)(a_1b_1+a_2b_2)}{b_2\omega_1^2(\omega_2^2-\omega_1^2)},\quad
\alpha_4=-\frac{a_1b_1(\omega_1^2-1)^2+a_2b_2(\omega_2^2-1)^2}
 {2b_2\omega_2^2(\omega_2^2-1)(\omega_2^2-\omega_1^2)},\\
 &
\beta_j=0,\quad
j=1,2,3,4.
\end{split}
\label{eqn:5b}
\end{equation}
See Appendix~A.2 for the derivation of \eqref{eqn:5b}.
Using \eqref{eqn:5brel},
 we see that the condition $\alpha_1\neq\alpha_3$ becomes
\begin{equation}
a_1b_1((\omega_2^2-1)^2-2b_1b_2)+a_2b_2((\omega_2^2-1)^2-2b_1b_2)\neq 0;
\label{eqn:5b1}
\end{equation}
$\alpha_2\alpha_3-\alpha_1\alpha_4\neq 0$ becomes
\begin{align}
&
(3a_1^2b_1^2+8a_1a_2b_1b_2+3a_2^2b_2^2-2)b_1^2b_2^2
-4b_1b_2(c-1)^2-(c-1)^4\neq 0;
\label{eqn:5b2}
\end{align}
$\alpha_2\alpha_4>0$ becomes
\begin{equation}
a_1b_1(\omega_1^2-1)^2+a_2b_2(\omega_2^2-1)^2>0;
\label{eqn:5b3}
\end{equation}
and $\alpha_2\alpha_4<0$ and $\alpha_2/\alpha_4\not\in\Qset$ become
\begin{equation}
a_1b_1(\omega_1^2-1)^2+a_2b_2(\omega_2^2-1)^2<0,\quad
\frac{(a_1b_1+a_2b_2)b_1b_2}
 {a_1b_1(\omega_1^2-1)^2+a_2b_2(\omega_2^2-1)^2}\not\in\Qset.
\label{eqn:5b4}
\end{equation}
Note that $\alpha_j\neq 0$ since $a_j,b_j>0$, $j=1,2$.
Applying Theorem~\ref{thm:main2}, we obtain the following.

\begin{prop}
\label{prop:5b}
When $\delta_j=0$, $a_j,b_j>0$, $j=1,2$, $b_1b_2<c$ and $\omega_1/\omega_2\not\in\Qset$,
 the coupled van der Pol oscillators \eqref{eqn:cvdp}
 are not real-analytically integrable near the origin
 if condition~\eqref{eqn:5b1} and one of conditions~\eqref{eqn:5b2}-\eqref{eqn:5b4} hold.
\end{prop}

\section*{Acknowledgments}
The author thanks Hidekazu Ito, Mitsuru Shibayama and Shoya Motonaga
 for their helpful discussions and useful comments.
This work was partially supported by the JSPS KAKENHI Grant Numbers
 JP17H02859 and JP22H01138.

\section*{Data Availability}
Data sharing not applicable to this article
 as no datasets were generated or analyzed during the current study.


\appendix

\renewcommand{\theequation}{\Alph{section}.\arabic{equation}}
\setcounter{equation}{0}

\section{Derivation of \eqref{eqn:5a} and \eqref{eqn:5b}}

We compute the coefficients in \eqref{eqn:pd1t} and \eqref{eqn:pd2t}
 for the examples in Sections~5.1 and 5.2,
 and derive \eqref{eqn:5a} and \eqref{eqn:5b}.
For the reader's convenience,
 we also give general formulas for computing these coefficients. 
See Sections~8.7.5 and 8.7.6 of \cite{K04} for the details.
We write the Taylor expansion of $f(x)$ around $x=0$ as
\[
f(x)=Ax+\tfrac{1}{2}B(x,x)+\tfrac{1}{6}C(x,x,x)+O(|x|^4),
\]
where $A=\D f(0)$,
 and $B(\xi,\eta)$ and $C(\xi,\eta,\zeta)$ are the bilinear and trilinear vector-values functions
 with components
\[
B_j(\xi,\eta)
=\sum_{k,l=1}^n\frac{\partial^2 f_j}{\partial x_k\partial x_l}(0)\xi_k\eta_l,\quad
C_j(\xi,\eta,\zeta)
 =\sum_{k,l,m=1}^n\frac{\partial^3 f_j}{\partial x_k\partial x_l\partial x_m}(0)
 \xi_k\eta_l\zeta_m
\]
for $ j=1,\ldots,n$ with $n=3$ or $4$.

\subsection{Derivation of \eqref{eqn:5a}}
Assume that $f(x)$ satisfies condition (I).
Let $v_0\in\Rset^3$ and $v_1\in\Cset^3$ be eigenvectors of $A$
 corresponding to the eigenvalues $\lambda=0$ and $i\omega$, respectively,
\[
Av_0=0,\quad
Av_1=i\omega v_1,
\]
and let $u_0\in\Rset^3$ and $u_1\in\Cset^3$ be eigenvectors of $A^\T$
 corresponding to the eigenvalues $\lambda=0$ and $-i\omega$, respectively,
\[
A^\T u_0=0,\quad
A^\T u_1=-i\omega u_1,
\]
such that $\langle u_0,v_0\rangle=\langle u_1,v_1\rangle=1$,
 where  $\langle\cdot,\cdot\rangle$ represents the inner product in $\Cset^n$.
Then we can transform \eqref{eqn:sys} to
\begin{equation}
\begin{split}
&
\dot{w}_0=\kappa_{01}w_0^2+\kappa_{02}|w_1|^2+O(|w_0|^3+|w_1|^3),\\
&
\dot{w}_1=i\omega w_1+\kappa_{11} w_0w_1+O(|w_0|^3+|w_1|^3),
\end{split}
\label{eqn:a1}
\end{equation}
where $w_0=\langle u_0,x\rangle\in\Rset$, $w_1=\langle u_1,x\rangle\in\Cset$ and
\begin{equation}
\begin{split}
&
\kappa_{01}=\tfrac{1}{2}\langle u_0,B(v_0,v_0)\rangle\in\Rset,\quad
\kappa_{02}=\langle u_0,B(v_1,v_1^\ast)\rangle\in\Rset,\\
&
\kappa_{11}=\langle u_1,B(u_0,v_1)\rangle\in\Cset
\end{split}
\label{eqn:a1c}
\end{equation}
with the superscript `$\ast$' denoting complex conjugate.
Letting $x_1=\Re\,w_1$, $x_2=\Im\,w_1$ and $x_3=w_0$,
 we rewrite \eqref{eqn:a1} as \eqref{eqn:pd1t} with 
\[
\alpha_1=\Re\,\kappa_{11},\quad
\alpha_2=\Im\,\kappa_{11},\quad
\alpha_3=\kappa_{02},\quad
\alpha_4=\kappa_{01}
\]
up to $O(|x|^2)$.

We now  compute the coefficients $\alpha_j$, $j=1$-$4$, for \eqref{eqn:cvdp}
 when $b=1$ and $c=a\in(-\sqrt{2},\sqrt{2})$.
We have
\[
A=
\begin{pmatrix}
0 & -1 & -1\\
1 & a & 0\\
1 & 0 & -a
\end{pmatrix},\quad
B(\xi,\eta)=
\begin{pmatrix}
0\\
0\\
\xi_1\eta_3+\xi_3\eta_1
\end{pmatrix}
\]and
\begin{align*}
&
v_0=
\begin{pmatrix}
a\\
-1\\
1
\end{pmatrix},\quad
v_1=
\begin{pmatrix}
a+i\omega\\
1-a^2-i\omega a\\
1
\end{pmatrix},\\
&
u_0=\frac{1}{\omega^2}
\begin{pmatrix}
-a\\
-1\\
1
\end{pmatrix},\quad
u_1=\frac{1}{2\omega^2(a^2-1-i\omega a)}
\begin{pmatrix}
a-i\omega\\
a^2-1-i\omega a\\
-1
\end{pmatrix}.
\end{align*}
By \eqref{eqn:a1c} we obtain
\begin{align*}
\kappa_{01}=\frac{a}{\omega^2},\quad
\kappa_{02}=\frac{2a}{\omega^2},\quad
\kappa_{11}=
 \frac{-a^3+i\omega(a^2+1)}{2\omega^2}
\end{align*}
which yields \eqref{eqn:5a}.

\subsection{Derivation of \eqref{eqn:5b}}

Assume that $f(x)$ satisfies condition (II).
For simplicity we also assume that $B(x,x)\equiv 0$.
For $j=1,2$, let $v_j\in\Cset^4$ be an eigenvector of $A$
 corresponding to the eigenvalue $\lambda=i\omega_j$,
\[
Av_j=i\omega_j v_j,
\]
and let $u_j\in\Cset^4$ be an eigenvector of $A^\T$
 corresponding to the eigenvalue  $-i\omega_j$,
\[
A^\T u_j=-i\omega_j u_j,
\]
such that $\langle u_j,v_j\rangle=1$.
Then we can transform \eqref{eqn:sys} to
\begin{equation}
\begin{split}
&
\dot{w}_1=i\omega_1w_1+\kappa_{11}w_1|w_1|^2+\kappa_{12}w_1|w_2|^2
 +O(|w_1|^4+|w_1|^4),\\
&
\dot{w}_2=i\omega_2 w_2+\kappa_{21}w_2|w_1|^2+\kappa_{21}w_2|w_1|^2
 +O(|w_1|^4+|w_1|^4),
\end{split}
\label{eqn:a2}
\end{equation}
where $w_j=\langle u_j,x\rangle\in\Cset$, $j=1,2$, and
\begin{equation}
\begin{split}
&
\kappa_{11}=\tfrac{1}{2}\langle u_1,C(v_1,v_1,v_1^\ast)\rangle,\quad
\kappa_{12}=\langle u_1,C(v_1,v_2,v_2^\ast)\rangle,\\
&
\kappa_{21}=\langle u_2,C(v_1,v_1^\ast,v_2)\rangle,\quad
\kappa_{22}=\tfrac{1}{2}\langle u_2,C(v_2,v_2,v_2^\ast)\rangle
\end{split}
\label{eqn:a2c}
\end{equation}
are also complex.
Letting $x_1=\Re\,w_1$, $x_2=\Im\,w_1$, $x_3=\Re\,w_2$ and $x_4=\Im\,w_2$,
 we rewrite \eqref{eqn:a2} as \eqref{eqn:pd2} with 
\begin{align*}
&
\alpha_1=\Re\,\kappa_{11},\quad
\alpha_2=\Re\,\kappa_{12},\quad
\alpha_3=\Re\,\kappa_{21},\quad
\alpha_4=\Re\,\kappa_{22},\\
&
\beta_1=\Im\,\kappa_{11},\quad
\beta_2=\Im\,\kappa_{12},\quad
\beta_3=\Im\,\kappa_{21},\quad
\beta_4=\Im\,\kappa_{22}
\end{align*}
up to $O(|x|^3)$.

We now  compute the coefficients $\alpha_j,\beta_j$, $j=1$-$4$,
 for \eqref{eqn:rsys} with $\delta_j=0$, $j=1,2$, and $b_1b_2<c$.
Note that $B(\xi,\eta)\equiv0$.
We have
\[
A=
\begin{pmatrix}
0 & 1 & 0 & 0\\
-1 & 0 & b_1 & 0\\
0 & 0 & 0 & 1\\
b_2 & 0 & -c & 0
\end{pmatrix},\quad
C(\xi,\eta,\zeta)=
\begin{pmatrix}
0\\
-2a_1(\xi_1\eta_1\zeta_2+\xi_1\eta_2\zeta_1+\xi_2\eta_1\zeta_1)\\
0\\
-2a_2(\xi_3\eta_3\zeta_4+\xi_3\eta_4\zeta_3+\xi_4\eta_3\zeta_3)
\end{pmatrix}
\]
and
\begin{align*}
&
v_1=
\begin{pmatrix}
ib_1/(\omega_1(\omega_1^2-1))\\
-b_1/(\omega_1^2-1)\\
-i/\omega_1\\
1
\end{pmatrix},\quad
u_1=\frac{\omega_1^2-1}{2(\omega_1^2-\omega_2^2)}
\begin{pmatrix}
-i\omega_1(\omega_2^2-1)/b_1\\
(\omega_2^2-1)/b_1\\
-i\omega_1\\
1
\end{pmatrix},\\
&
v_2=
\begin{pmatrix}
ib_1/(\omega_2(\omega_2^2-1))\\
-b_1/(\omega_2^2-1)\\
-i/\omega_2\\
1
\end{pmatrix},\quad
u_2=\frac{\omega_2^2-1}{2(\omega_2^2-\omega_1^2)}
\begin{pmatrix}
-i\omega_2(\omega_1^2-1)/b_1\\
(\omega_1^2-1)/b_1\\
-i\omega_2\\
1
\end{pmatrix}.
\end{align*}
Using \eqref{eqn:5brel} and \eqref{eqn:a2c}, we obtain
\begin{align*}
&
\kappa_{11}=\frac{a_1b_1(\omega_2^2-1)^2+a_2b_2(\omega_1^2-1)^2}
 {2b_2\omega_1^2(\omega_1^2-1)(\omega_2^2-\omega_1^2)},\\
&
\kappa_{12}=\frac{(\omega_1^2-1)(a_1b_1+a_2b_2)}{b_2\omega_2^2(\omega_2^2-\omega_1^2)},\\
&
\kappa_{21}=-\frac{(\omega_2^2-1)(a_1b_1+a_2b_2)}{b_2\omega_1^2(\omega_2^2-\omega_1^2)},\\
&
\kappa_{22}=-\frac{a_1b_1(\omega_1^2-1)^2+a_2b_2(\omega_2^2-1)^2}
 {2b_2\omega_2^2(\omega_2^2-1)(\omega_2^2-\omega_1^2)},
\end{align*}
which yields \eqref{eqn:5b}.


\end{document}